\documentclass[leqno,12pt]{amsart} 
\title[New Lie algebras over $\bb Z_2^3$]{New  Lie algebras over the group $\bb Z_2^3$}

\author[]{Francisco Cuenca Carr\'egalo$^{\ast}$   and Cristina Draper$^{\star}$ 
}

\usepackage{amssymb,color, bm}
\usepackage{stackrel}
\subjclass[2010]{Primary  
17B70; %: Graded Lie (super)algebras
Secondary 
17B30. %Solvable, nilpotent (super)algebras
}

\keywords{Generalized group algebras, Lie algebras over   $\bb Z_2^3$, group gradings, twisted group algebras, graded contractions,  orthogonal algebras,  solvable and nilpotent Lie algebras, octonions}
\thanks{$^{\star}\,$Supported by Junta de Andaluc\'{\i}a  through project  FQM-336,   and  by the Spanish Ministerio de Ciencia e Innovaci\'on   through projects PID2020-118452GB-I00 and PID2023-152673GB-I00, with FEDER funds. $^{\ast}$ Supported by PID2023-152673GB-I00.}

\usepackage{amssymb, mathrsfs, amsmath, tabu, amsthm}
\usepackage{hyperref}
\usepackage{tikz-cd, graphicx} 
\usepackage{pdfpages}
\usepackage{enumitem}
\usepackage{ragged2e}
\usetikzlibrary{decorations.markings}

\hypersetup{
	colorlinks = true,
	urlcolor = blue,
	linkcolor = blue,
	citecolor = red
}
\newcommand{\la}{\mathfrak{L}} % Easy way to write gothic 'L'
\newcommand{\f}[1]{\mathfrak{#1}}
\newcommand{\supp}[1]{\Supp^{#1}} 
\newcommand{\bb}[1]{\mathbb{#1}}
\newcommand{\ang}[1]{\langle #1 \rangle}
\newcommand{\ep}{\varepsilon}

\newcommand{\comment}[1]{}

\DeclareMathOperator{\Supp}{\mathcal{S}}

\DeclareMathOperator{\ad}{ad}

\theoremstyle{plain} %text of this environment is typesetted in italics
\newtheorem{theorem}{Theorem}[section]
\newtheorem{lemma}[theorem]{Lemma}
\newtheorem{cor}[theorem]{Corollary}
\newtheorem{prop}[theorem]{Proposition}

\theoremstyle{definition} %text of this environment is typesetted in roman letters
\newtheorem{define}[theorem]{Definition}

\newtheorem{remark}[theorem]{Remark}
\newtheorem{example}[theorem]{Example}

\begin{document}

%%%%%%%%%%%%%%%%%%%%%

\maketitle

\begin{abstract}
A new structure, based on joining copies of a group by means of a \emph{twist}, has recently  been considered to describe the brackets of the two exceptional real Lie algebras of type $G_2$ in a highly symmetric way. 
In this work we show that these are not  isolated examples, providing a wide range of Lie algebras which are generalized group  algebras over the group  $\mathbb{Z}_2^3$. 
On the one hand, some orthogonal Lie algebras are quite naturally generalized group   algebras over such group. 
On the other hand, previous classifications on graded contractions can be applied to this context getting many more examples, involving solvable and nilpotent Lie algebras of dimensions 32, 28, 24, 21, 16 and 14.

\end{abstract}
%\tableofcontents
	
 \section{Introduction }  
% This paper turns on two ideas. 
This paper revolves around two ideas, involving the two concepts of generalized group algebra and of graded contraction, both focussed on Lie algebras. 
A \emph{generalized group algebra} generalizes both  a twisted group algebra $\mathbb{F}^\sigma[G]$, for $\bb F$  a field and $\sigma\colon G\times G\to \bb F$ a map (the earliest works date from the 1960s, as \cite{Twistedgroupalgebras64}),  and a group ring   $R[G]$, in case  the ring $R$ is also a vector space over $\bb F$ (see, for instance, \cite{passman2011algebraic}). 
Although at first glance a generalized group algebra would appear to be a strange object, of which there are no or uninteresting examples, and above all with no apparent connection to Lie algebras, the   definition is motivated by a very remarkable example 
appeared in \cite{draper2024twistedg2}. That work  provided a nice and practical model of each of the two real Lie algebras of type 
 $G_2$, both the compact   and the split ones, as generalized group algebras over $\bb Z_2^3$. 
 In general, it is particularly difficult to find ways of describing the compact Lie algebra $\f g_{2,-14}$. Its usual description 
  as the Lie algebra of derivations of the octonion division algebra (or of the octonion split algebra, in the $\f g_{2,2}$-case),  
  is not   easy to use at all, because not even the elements have intuitive expressions as linear operators. 
 The mentioned models as generalized group algebras in  \cite{draper2024twistedg2} exploited the symmetry on the group $\bb Z_2^3$, providing the bracket in a self-contained way, without the need for prior knowledge of either octonions or their derivations. 
 A first question is whether the concept of generalized group algebra can be useful for describing further Lie algebras, or whether the two examples of $G_2$-type are isolated examples. 
 In this work, we will find generalized group algebras that are reductive (direct sum of a semisimple ideal  with the center),
   solvable Lie algebras, and nilpotent Lie algebras, displaying a considerable range of examples. In other words, the concept deserves further study. 
   
 The second goal is to apply   \emph{graded contractions} to find new Lie algebras with precise brackets. This concept comes from physics, from the long physics tradition of varying the product, either by complicating or simplifying it. 
 Take care with the fact that a graded contraction, as introduced in \cite{patera1991discrete},
  is not a contraction
which is graded, since it is defined algebraically and not by a limiting process.
More precisely, it consists of preserving the bracket between two homogeneous components up to scalar multiple depending on some \emph{contraction} $\sigma\colon G\times G\to \bb F$.  %
 In \cite{draper2024gradedb3d4}, graded contractions of the orthogonal Lie algebras $\f{so}_8(\bb C)$ and $\f{so}_7(\bb C)$ are completely classified up to equivalence (which means that the related Lie algebras are isomorphic), 
 obtaining   a large list of Lie algebras.
 The considered  $\bb Z_2^3$-gradings are not arbitrary, they are chosen because they share important properties with the $\bb Z_2^3$-grading on $\f g_2$ coming from octonions. 
 The point is that one cannot know the exact bracket in the contracted Lie algebras without knowing first a precise description of the original bracket in the orthogonal algebra adapted to the grading. 
 Once you have described the orthogonal algebras as generalized group algebras, you get as a bonus the concrete description as
 generalized group algebras of all the Lie algebras obtained  by graded contraction.
  This will greatly increase the number of examples available. 
 \smallskip

The structure of this work  follows. 
The definition of a generalized group  algebra is stated in Section~\ref{subsec_maindef}, as a generalization of a twisted group algebra when replacing the field $\bb F$ with $V$ a vector space over $\bb F$. This requires of a map 
$\sigma: G\times G \rightarrow \mathrm{Bil}(V\times V,V)$, called, by analogy, a twist.
As the definition is quite new, some examples follow.  It is not difficult to show, 
in Proposition~\ref{pr_d4} and Corollary~\ref{cor_casob3}, that  some orthogonal Lie algebras of size $8$ and $7$ are two more   examples of generalized group  algebras. 
They can be constructed by  copying 
convenient $\bb Z_2^3$-gradings on those orthogonal Lie algebras. 
Results on $\f g_2$ as Lie algebra over $\bb Z_2^3$ are recovered too. 
This approach of emphasizing the role of the group has several advantages, 
as shown in Sections~\ref{se_kappa} and \ref{se_rep}, which deal with the Killing form
and with the representations compatible with the gradings.
The second part of this work, developed in Section~\ref{se_3}, is focused on obtaining more examples of generalized group  algebras.
The background on graded contractions is recalled in 
 Section~\ref{subsec_prelim_gr_cont}, jointly with the classification up to equivalence of the graded contractions of our remarkable gradings   in Section~\ref{subsec_clasificacion}.
 The crucial Lemma~\ref{le_crucial} allows us to obtain  a large collection of generalized group  algebras in Corollary~\ref{cor_results}. 

 %%%%%%%%%%%%%%%%%%%%%%%%%%%%%%%%%%%%%%%%%%%%%%%%%%%%%%%%%%%%%%%%%%%%%%%%%%%%%%%%%%%%%%%%%%%%%%%%%%%%%%%%%%%%%%%%%%%%%%%%%%%%%
\section{Some orthogonal   algebras which are generalized group algebras}

Throughout this work, $\bb F$ will be an arbitrary field, most of times  of   characteristic different from 2 and 3.  

\subsection{Generalized group   algebras which are Lie algebras}\label{subsec_maindef}

For  $G$ an abelian group, and a map  $\sigma\colon G\times G\to \bb F$, 
the  \emph{twisted group algebra} $\bb F^\sigma[G]$ 
  consists of endowing $\bb F[G]=\{\sum_{g\in G}\alpha_gg:\alpha_g\in \bb F\}$   with the only product   defined by bilinear extension of  $g\cdot h=\sigma(g,h)(g+h)$.
This structure encompasses a wide range of examples of different types of algebras, for instance octonion algebra was described as $\mathbb{F}^\sigma[\mathbb{Z}_2^3]$ in \cite{albuquerque1999quiasialgebra}, Clifford algebras as $\mathbb{F}^\sigma[\mathbb{Z}_2^n]$ in \cite{elduque2018clifford}, and Albert algebra as $\mathbb{F}^\sigma[\mathbb{Z}_3^3]$ in \cite{elduque2012weyl} for convenient maps $\sigma$'s.
 When replacing the field $\bb F$   with a ring $R$, we can consider the  \emph{group ring} of $G$ with coefficients in $R$ again as the set of formal sums 
 $R[G]=\{\sum_{g\in G}r_gg:r_g\in R\}$ with  product   extending $(rg)\cdot(r'h)=(rr')(g+h)$. In both cases, two elements
  are considered equal if and only if the coefficients of each group element
are equal. 
In  \cite{draper2024twistedg2},  a mixture of the above two concepts   appears,  in principle for describing the smallest of the exceptional Lie algebras. 
The term that was used there, \emph{twisted ring  group} (Eq.~\eqref{eq_definciondelproducto} below) is probably   inadequate: it does not seem to fit our usage very well (with a strong focus on Lie algebras), nor does it seem to coincide with previous uses of the term (see the book \cite[Chapter 1, Section 2]{passman2011algebraic}).   The construction in \cite{draper2024twistedg2} is well suited to the following definition, where the only we need is a triple $(G,V,\sigma)$ with $V$  a vector space.

\begin{define}\label{def_twisted}
  
    Let $(G,+)$ be an abelian group, $V$  a vector space over $\bb F$,  and $\sigma: G\times G \rightarrow \mathrm{Bil}(V\times V,V),\,(g,h)\mapsto\sigma_{g,h}$ a map. We endow  the set of formal sums
    $$
    V^\sigma[G]:=\left\{ \sum_{g\in G}r_g g: r_g \in V \right\}
    $$ 
    with  an $\bb F$-algebra structure by means of $\alpha(r_gg):=(\alpha r_g)g$ and 
  \begin{equation}\label{eq_definciondelproducto}
    \left(\sum_{g\in G}r_g g \right) \cdot  \left(\sum_{h\in G}s_hh\right) :=  \sum_{g,h\in G} \sigma_{g,h}(r_g,s_h)(g+h),
  \end{equation} 
for $r_g,s_h\in V$, $g,h\in G$ and $\alpha \in \bb F$.
We refer to this algebra  $ V^\sigma[G]$ as  \emph{generalized group   algebra}, or simply \emph{algebra over $G$}, with the aim of emphasizing the role played by the concrete group. In case that the generalized group   algebra  $ V^\sigma[G]$ 
  turns out to be a Lie algebra with this product,   we will refer to it as a \emph{Lie algebra over $G$}, and the product \eqref{eq_definciondelproducto} will be written with a bracket. Sometimes, we will talk about $\sigma$ as a \emph{twist}\footnote{The term \emph{twist} is inspired in Example~\ref{examples}(1), although in the general case, we are not twisting any previous product in $V$, rather,  each $\sigma_{g,h}$ endows the vector space  $V$ with a ring structure.}.
\end{define}

\begin{example}\label{examples}
\begin{enumerate}
\item Any {twisted group algebra} $\bb F^\sigma[G]$  is a  generalized group algebra for $V=\bb F$ and
 $\sigma\colon G\times G\to \bb F$, where we  trivially  identify $\mathrm{Bil}(\bb F\times \bb F,\bb F)$  with $\bb F$  by assigning to a bilinear map the image of $(1,1)$. 
 Here the general conditions for the generalized group   algebra to be a Lie algebra turn out to be 
 $$
 \sigma(g,h)=-\sigma(h,g),\qquad \sum_{g,h,k\ \mathrm{ cyclic}}\sigma(g,h)\sigma(g+h,k)=0,
 $$
 which translate the skew-symmetry and the Jacobi identity respectively. 
 For instance, if $\sigma\equiv 0$, $\bb F^\sigma[G]$ is an abelian Lie algebra.
 There are   nontrivial occurrences too, as the next example shows.
\item The general linear algebra $\f{gl}_n(\bb F)=(\mathrm{Mat}_{n\times n}(\bb F),[\cdot ,\cdot ])$ is a generalized group   algebra  for any algebraically closed field $\bb F$, $V=\bb F$, $G=\bb Z_n^2$ and 
$$
\sigma\colon G\times G\to \mathrm{Bil}(\bb F\times \bb F,\bb F)\equiv \bb F, \quad
\sigma((\overline{a_1},\overline{a_2}),(\overline{b_1},\overline{b_2}))=\xi^{a_2b_1}-\xi^{a_1b_2},
$$
where $\xi$ is a primitive $n$th root of the unity. For proving this, recall that $\f{gl}_n(\bb F)$ is linearly spanned by the set $\{X^aY^b:a,b=0,\dots,n-1\}$ for
\begin{equation*}
X=\left(\begin{array}{ccccc}1&0&\dots&\dots&0\\0&\xi&0&\dots&0
\\\vdots&\ddots&\ddots&\ddots&\vdots\\0&\ldots&0&\xi^{n-2}
&0\\0&\ldots&\ldots&0&\xi^{n-1}\end{array}\right),
\quad
Y=\left(\begin{array}{ccccc}0&1&0&\dots&0\\0&0&1&\dots&0
\\{\vdots}&\ddots&\ddots&\ddots&\vdots\\0&\ldots&0&0&1\\1&0&\ldots&\ldots&0
\end{array}\right).
\end{equation*}
As $X^n=I_n=Y^n$ and $YX=\xi XY$, then 
$$[X^{a_1}Y^{a_2},X^{b_1}Y^{b_2}]=(\xi^{a_2b_1}-\xi^{a_1b_2})X^{a_1+b_1}Y^{a_2+b_2}.
$$
Then the  identification $X^{a_1}Y^{a_2}\mapsto (\overline{a_1},\overline{a_2})\in G$ gives the required isomorphism.
\item Both the real algebras $\f g_{2,-14}$ and $\f g_{2,2}$ are simple ideals of generalized group   algebras  for $\bb F=\bb R$, $V=\bb F^2$, $G=\bb Z_2^3$ and the explicit $\sigma$'s described  in \cite[Theorem~1 and Corollary~3]{draper2024twistedg2}, respectively. 
To be precise, those exceptional real Lie algebras appear as $V^\sigma[G^\times]$, removing the neutral element of the group.
In order to include these and other cases of interest,  subalgebras of Lie    algebras over $G$ will be also called   Lie    algebras over $G$.
Note also that, if we consider the complex field  $\bb F=\bb C$, both the obtained Lie algebras become isomorphic to  the only complex Lie algebra of type $G_2$, another Lie algebra over $\bb Z_2^3$. 
%.....a lo mejor lo pongo luego para contrastar
\end{enumerate}  
\end{example}

 %%%%%%%%%%%%%%%%%%%%%%%%%%%%%%%%%%%%%%%%%%%%%%%%%%%%%
\subsection{The orthogonal Lie algebra  of size   8 as generalized group   algebra}\label{subsec_d4istwisted}

From now on through this work, $G$ will always be $\bb Z_2^3$.
% and $R=\bb F^n$ with   componentwise product, written by $*$. 
The elements in $G=\bb Z_2^3$ admit a labelling
\begin{equation}
    \label{eq: elementos de Z_2^3}
    \begin{array}{cccc}
        g_0:=(\bar{0},\bar{0},\bar{0}), & g_1:=(\bar{1},\bar{0},\bar{0}), & g_2:=(\bar{0},\bar{1},\bar{0}), & g_3:=(\bar{0},\bar{0},\bar{1}), \\
        g_4:=(\bar{1},\bar{1},\bar{0}), & g_5:=(\bar{0},\bar{1},\bar{1}), & g_6:=(\bar{1},\bar{1},\bar{1}), & g_7:=(\bar{1},\bar{0},\bar{1}), 
    \end{array}
\end{equation}
such that $g_{i}+g_{i+1}=g_{i+3}$ for any $i\in I=\{1,\dots,7\}$, where the sum of indices is considered modulo 7.
(Hence $g_{i+1}+g_{i+3}=g_{i}$ and $g_{i+3}+g_{i}=g_{i+1}$.)
For further use, denote by $i*j$ the only index in $I\cup\{0\}$ such that $g_{i}+g_{j}=g_{i*j}$.

A $G$-\emph{grading} $\Gamma$ on an algebra  $\mathcal{A}$    is a vector space decomposition $\Gamma: \mathcal{A} = \bigoplus_{g\in G} \mathcal{A}_g$ such that $\mathcal{A}_g \mathcal{A}_h \subset \mathcal{A}_{g + h}$, for all $g, \, h\in G$.  Each subspace $\mathcal{A}_g$ is called a \emph{homogeneous component}, and $g$ its \emph{degree}.  
Generalized group   algebras $V^\sigma[G]$ are naturally $G$-graded, for 
$(V^\sigma[G])_g=Vg$. 

In particular, any Lie algebra over $G$ is $G$-graded. 
(Gradings over groups    have proved to be a key tool to study Lie algebras, so the readers are kindly invited to consult the AMS monography \cite{elduque2013gradings_simple_lie} for further information on gradings on simple Lie algebras.)
Of course the converse does not necessarily occur, and a necessary condition  for a $G$-graded Lie algebra to be a 
Lie  algebra over $G$ is that all the homogeneous components have the same dimension (coinciding of course with the dimension of $V$ as a vector space over $\bb F$). 

Our main aim in this section is to describe the orthogonal Lie algebra $ \f{so}_8(\bb F)$ of the skew-symmetric matrices 
as a simple ideal of a generalized group  algebra over the group $\bb Z_2^3$, based on our knowledge of a concrete $\bb Z_2^3$-grading on $ \f{so}_8(\bb F)$ with strong symmetry properties, including that one of constant dimension of the homogeneous components (up to the neutral component, which vanishes). The main relevant facts concerning this grading can be extracted for instance from \cite{draper2024gradedb3d4}, but we will recall them here  for fixing the notation.

Denote by $\ang{\cdot,\cdot}$ the usual scalar product in $\bb F^8$, that is, $\ang{\sum_{i=0}^7x_ie_i,\sum_{i=0}^7y_ie_i}=\sum_{i=0}^7x_iy_i$, for $\{e_i:i=0,\dots,7\}$ the canonical basis of $\bb F^8$.
The orthogonal algebra  
$$
\f{so}_8(\bb F):=\{ f \in \f{gl}_8(\bb F): \ang{f(x),y}+\ang{x,f(y)}=0 \text{ for any }x,y\in\mathbb{F}^8\},
$$ 
is spanned by the linear operators $\varphi_{ij}\equiv \varphi_{e_i,e_j}$, where
$$
\varphi_{x,y}\colon \mathbb{F}^8\to \mathbb{F}^8,\quad 
\varphi_{x,y}(z):=\ang{x,z}y-\ang{y,z}x.
$$
The $G$-grading on the vector space $V=\bb F^8$  (no more than a vector space decomposition labelled on the group) obtained by assigning $\deg(e_i)=g_i\in G$, induces a $G$-grading on the orthogonal algebra, 
denoted as $\Gamma_{\f d_4}:\f{so}_8(\bb F)=\oplus_{g\in G} \f{so}_8(\bb F)_g$, 
%$ \f{so}_8(\bb F)$  as usual, 
in the usual way,
 i.e., a map $f \in \f{so}_8(\bb F)$ has degree $g$ if $f(V_h)\subset V_{g+h}$ for all $h\in G$. 
%$\varphi_{ij}\colon \mathbb{F}^8\to \mathbb{F}^8$, $\varphi_{ij}(z)=\ang{e_i,z}e_j-\ang{e_j,z}e_i$.
Taking into account that $\varphi_{ij}(e_k)=\delta_{ik}e_j-\delta_{jk}e_i$ ($\delta$ is used for the Kronecker delta), then $\varphi_{ij}\in  \f{so}_8(\bb F)_{g_i+g_j}$. 
As $g_{i+2}+g_{i+6}=g_{i}=g_{i+4}+g_{i+5}$ for all $i\in I$, in particular
%From here, it is clear that
$$
\mathcal{B}_i=\{\varphi_{{i+1}, {i+3}}, \varphi_{{i+2}, {i+6}}, \varphi_{{i+4}, {i+5}}, \varphi_{0, i}\}\subset  \f{so}_8(\bb F)_{g_i}.
$$
By dimension count, this implies that   $\mathcal{B}_i$ is a basis of $ \f{so}_8(\bb F)_{g_i}$ if $i\ne0$ and that $ \f{so}_8(\bb F)_{g_0}=0$. 
Furthermore, each homogeneous component $ \f{so}_8(\bb F)_{g_i}$ is an abelian subalgebra, since, for any $\psi\in  \f{so}_8(\bb F)$, and any pair of elements  $x,y\in\bb F^8$, the next identity holds
 \begin{equation}\label{eq_lunaazul}
 [\psi,\varphi_{x,y}]=\varphi_{\psi(x),y}+\varphi_{x,\psi(y)},
 \end{equation}
 from which it is easy to deduce  $[\mathcal{B}_i,\mathcal{B}_i]=0.$
A graded Lie algebra where all the homogeneous components (up to the neutral component) are abelian and have the same dimension is the key to endow it with a  structure of generalized group  algebra over the grading group. In our case, denote by $G^\times=G\setminus\{g_0\}$ and 
simply consider the vector space isomorphism
\begin{equation}
        \label{eq: isomorfismo entre d4 y el twisted}
        \begin{array}{cccl}
            \Psi: & \f{so}_8(\bb F) & \longrightarrow & V[G^\times]  \\
             &a_1\varphi_{{i+1}, {i+3}}+a_2 \varphi_{{i+2}, {i+6}}+a_3 \varphi_{{i+4}, {i+5}}+a_4  \varphi_{0, i}  &\longmapsto &(a_1,a_2,a_3,a_4) g_i ,
        \end{array}
    \end{equation}
    which is a graded map ($\Psi( \f{so}_8(\bb F)_{g_i})\subset Vg_i$).
    Now   define the bracket in the subspace $V[G^\times] $ (in other words, take $\sigma$) in the only way that makes $ \Psi\colon  \f{so}_8(\bb F)\to V^\sigma[G^\times] $ a Lie algebra isomorphism. For consistency with  Definition~\ref{def_twisted}, we extend $\sigma$  to the whole group $G$.
    The appropriate $\sigma$ follows:

\begin{prop}\label{pr_d4}
Let $V=\bb F^4$ and $G=\bb Z_2^3$. Then $ \la=V^\sigma[G]$ is a generalized group  algebra for the map 
$\sigma\colon G\times G \rightarrow \mathrm{Bil}(V\times V,V)$, $\sigma(g_i,g_j)\equiv \sigma_{i,j}$ given by, for any  $i\in I$,
\begin{equation}\label{eq_hola}
\begin{array}{l}
\sigma_{i,i+1}(r,s)=
 (-a_2b_1-a_3b_3,-a_2b_3-a_3b_1,a_1b_2+a_4b_4,a_1b_4+a_4b_2),\\
         \sigma_{i,i+2}(r,s)= (a_2b_3+a_4b_4,-a_1b_1-a_3b_2,-a_1b_2-a_3b_1,a_2b_4+a_4b_3),  \\ 
        \sigma_{i,i+4}(r,s)= (-a_1b_2-a_2b_3,a_3b_1+a_4b_4,-a_1b_3-a_2b_2,a_3b_4+a_4b_1),
\end{array}
\end{equation}
if $r=(a_1,a_2,a_3,a_4)$, $s=(b_1,b_2,b_3,b_4)$, 
$
\sigma_{0i}(r,s)=\sigma_{i0}(r,s)=\sigma_{00}(r,s)=\sigma_{ii}(r,s)=0
$
 and
$$
 \sigma_{i,i+3}(r,s)=-\sigma_{i,i+4}(s,r),\quad
  \sigma_{i,i+5}(r,s)=-\sigma_{i,i+2}(s,r),\quad
   \sigma_{i,i+6}(r,s)=-\sigma_{i,i+1}(s,r).
$$
Furthermore the center $\f z(\la)=\{rg_0:r\in V \}$ has dimension 4, and the derived algebra
$[\la,\la]=\langle\{rg_i:r\in V,i\in I \}\rangle\equiv V^\sigma[G^\times]$ is simple and isomorphic to $\f{so}_8(\bb F)$.
\end{prop}

\begin{proof}
 It is clear that $[rg_0,V^\sigma[G] ]=0$, so that $Vg_0$ is central and
 we have only to check that   the map $ \Psi\colon  \f{so}_8(\bb F)\to V^\sigma[G^\times] \le\la$,  defined in Eq.~\eqref{eq: isomorfismo entre d4 y el twisted}, is an algebra isomorphism (in particular $V^\sigma[G^\times] $ would be a Lie algebra). Thus, 
 let us check that $\Psi([f,f'])=[\Psi(f),\Psi(f')]$ for any homogeneous elements  $f,f'\in  \f{so}_8(\bb F)$ of degrees $i,j\in I$, respectively. For
    $$\begin{array}{ll}
    f= a_1\varphi_{{i+1}, {i+3}}+a_2 \varphi_{{i+2}, {i+6}}+a_3 \varphi_{{i+4}, {i+5}}+a_4\varphi_{0, i},&r=(a_1,a_2,a_3,a_4),\\
    f'= b_1\varphi_{{j+1}, {j+3}}+b_2 \varphi_{{j+2}, {j+6}}+b_3 \varphi_{{j+4}, {j+5}}+b_4\varphi_{0, j},&r'=(b_1,b_2,b_3,b_4),
    \end{array}
    $$
    we have 
    $
    [\Psi(f),\Psi(f')]=[rg_i,r'g_{j}]=\sigma_{i,j}(r,r')g_{i*j},
    $
    and we want to check that this coincides with $\Psi ([f,f'])$.
  If $i=j$, no problem arises since $ \f{so}_8(\bb F)_{g_i}$ is   abelian  and $\sigma_{ii}\equiv0$.
 Compute, for $j=i+1$, the brackets of basic elements with the help of \eqref{eq_lunaazul},
     \begin{center}  
    \begin{tabular}{|c|cccc|}
        \hline 
         $[\cdot,\cdot]$ & $\varphi_{ {j+1},  {j+3}}$ & $\varphi_{ {j+2},  {j+6}}$& $\varphi_{ {j+4},  {j+5}}$& $ \varphi_{ 0,  j}$ \\
        \hline
        $\varphi_{ {i+1},  {i+3}}$ & $0$ & $\varphi_{ {i*j+4}, {i*j+5}}$ & $0$ & $\varphi_{ 0, {i*j}}$\\
        $\varphi_{ {i+2},  {i+6}}$ & $-\varphi_{ {i*j+1}, {i*j+3}}$ & $0$ & $-\varphi_{ {i*j+2}, {i*j+6}}$ &$ 0$\\
        $\varphi_{ {i+4},  {i+5}}$ & $-\varphi_{ {i*j+2}, {i*j+6}}$ & $0$ & $-\varphi_{ {i*j+1}, {i*j+3}}$ & $0$\\
        $\varphi_{ 0,  i}$ & $0$ & $\varphi_{ {0}, {i*j}}$ & $0$ & $\varphi_{ {i*j+4}, {i*j+5}}$ \\
        \hline 
    \end{tabular}
    \end{center}
This immediately gives 
    $$
    [f,f']=(-a_2b_1-a_3b_3)\varphi_{i+4,i+6}+(-a_2b_3-a_3b_1)\varphi_{i+5,i+2}+(a_1b_2+a_4b_4)\varphi_{i,i+1}+(a_1b_4+a_4b_2)\varphi_{0,i+3}
    $$
  so that
    $$
    \Psi([f,f'])=(-a_2b_1-a_3b_3,-a_2b_3-a_3b_1,a_1b_2+a_4b_4,a_1b_4+a_4b_2)g_{i+3}=\sigma_{i,i+1}(r,r')g_{i*(i+1)}.
    $$
    Similarly we have to compute the brackets among basic elements for the case $j=i+2$,  
     \begin{center}  
    \begin{tabular}{|c|cccc|}
        \hline 
         $[\cdot,\cdot]$ & $\varphi_{ {j+1},  {j+3}}$ & $\varphi_{ {j+2},  {j+6}}$& $\varphi_{ {j+4},  {j+5}}$& $ \varphi_{ 0,  j}$ \\
        \hline
        $\varphi_{ {i+1},  {i+3}}$ & $-\varphi_{i*j+2,i*j+6}$ &$-\varphi_{i*j+4,i*j+5}$ &0&0\\
        $\varphi_{ {i+2},  {i+6}}$ &0 &0&$\varphi_{i*j+1,i*j+3}$ &$\varphi_{0,i*j}$ \\
        $\varphi_{ {i+4},  {i+5}}$ &$-\varphi_{i*j+4,i*j+5}$ &$-\varphi_{i*j+2,i*j+6}$ &0&0 \\
        $\varphi_{ 0,  i}$ & 0 &0&$\varphi_{0,i*j}$ &$\varphi_{i*j+1,i*j+3}$  \\
        \hline 
    \end{tabular}
    \end{center}
    and  in case $j=i+4$,
     \begin{center}  
    \begin{tabular}{|c|cccc|}
        \hline 
         $[\cdot,\cdot]$ & $\varphi_{ {j+1},  {j+3}}$ & $\varphi_{ {j+2},  {j+6}}$& $\varphi_{ {j+4},  {j+5}}$& $ \varphi_{ 0,  j}$ \\
        \hline
        $\varphi_{ {i+1},  {i+3}}$ & 0 &$-\varphi_{i*j+1,i*j+3}$ &$-\varphi_{i*j+4,i*j+5}$ &0\\
        $\varphi_{ {i+2},  {i+6}}$ &0 &$-\varphi_{i*j+4,i*j+5}$ &$-\varphi_{i*j+1,i*j+3}$ &0\\
        $\varphi_{ {i+4},  {i+5}}$ &$\varphi_{i*j+2,i*j+6}$ &0&0&$\varphi_{0,i*j}$  \\
        $\varphi_{ 0,  i}$ &  $\varphi_{0,i*j}$ &0&0&$\varphi_{i*j+2,i*j+6}$  \\
        \hline 
    \end{tabular}
    \end{center}
The cases $j=i+6,i+5,i+3$ are consequence of the skew-symmetry in $\f{so}_8(\bb F)$
%(of $\f{so}_8(\bb F)$ at the moment)
and of the   fact $\sigma_{ij}(r,s)=-\sigma_{ji}(s,r)$. For instance, if $j=i+6$, as $i=(i+6)+1$,
\begin{multline}
 [\Psi(f),\Psi(f')]=\sigma_{i,i+6}(r,r')(g_i+g_{i+6})=-\sigma_{i+6,i}(r',r)g_{i+2}\\
 =-\sigma_{i,i+1}(r',r)( g_{i+6}+g_{i})=-[\Psi(f'),\Psi(f)]= -\Psi([f',f])=\Psi([f,f']).
 \end{multline}
 This finishes the proof.
\end{proof}

 %%%%%%%%%%%%%%%%%%%%%%%%%%%%%%%%%%%%%%%%%%%%%%%%%%%%%
\subsection{The series $\f{g}_2\subset \f{b}_3\subset \f{d}_4 $ as Lie   algebras over $\bb Z_2^3$}\label{subsec_b3istwisted}

From the above construction of the orthogonal algebra of size 8 as generalized group algebra, we can deduce that   the orthogonal 
algebra of size 7 is another example of generalized group algebra, and then recover the result of \cite{draper2024twistedg2} 
which says that a certain subalgebra of type $G_2$   can be described too as a Lie   algebra over $\bb Z_2^3$.   

First, if we consider the vector subspace $V'=\{(a,b,c,0):a,b,c\in \bb F\}\le V$, it is clear that $V'[G]=\left\{ \sum_{g\in G}r_g g: r_g \in V' \right\}$ is closed for the bracket considered in
Proposition~\ref{pr_d4}, hence it is a Lie subalgebra of 
$V^\sigma[G]$. This is precisely the sum of an orthogonal algebra of size 7 (simple Lie algebra of type $B_3$) with a 3-dimensional center:

\begin{cor}\label{cor_casob3}
Let $W=\bb F^3$ and $G=\bb Z_2^3$. Then $ \la=W^\sigma[G]$ is a generalized group  algebra for the map 
$\sigma\colon G\times G \rightarrow \mathrm{Bil}(W\times W,W)$, $\sigma(g_i,g_j)\equiv \sigma_{i,j}$ given by, for any  $i\in I$,
\begin{equation}\label{eq_hola2}
\begin{array}{l}
\sigma_{i,i+1}(r,s)=
 (-a_2b_1-a_3b_3,-a_2b_3-a_3b_1,a_1b_2),\\
         \sigma_{i,i+2}(r,s)= (a_2b_3,-a_1b_1-a_3b_2,-a_1b_2-a_3b_1),  \\ 
        \sigma_{i,i+4}(r,s)= (-a_1b_2-a_2b_3,a_3b_1,-a_1b_3-a_2b_2),
\end{array}
\end{equation}
if $r=(a_1,a_2,a_3)$, $s=(b_1,b_2,b_3)$, 
$
\sigma_{0i}(r,s)=\sigma_{i0}(r,s)=\sigma_{00}(r,s)=\sigma_{ii}(r,s)=0
$
 and
$$
 \sigma_{i,i+3}(r,s)=-\sigma_{i,i+4}(s,r),\quad
  \sigma_{i,i+5}(r,s)=-\sigma_{i,i+2}(s,r),\quad
   \sigma_{i,i+6}(r,s)=-\sigma_{i,i+1}(s,r).
$$
Furthermore the center $\f z( \la)=\{rg_0:r\in W \}$ has dimension 3, and the derived algebra
$[\la,\la]=\langle\{rg_i:r\in W,i\in I \}\rangle\equiv W^\sigma[G^\times]$ is simple and isomorphic to $\f{so}_7(\bb F)$.
\end{cor}

\begin{proof}
The fact that $ \la=W^\sigma[G]$ is a generalized group  algebra for our choice of the twist  $\sigma$ was justified before the proof:
the concrete expression comes from substituting $a_4=b_4=0$ in the twist considered  in 
Proposition~\ref{pr_d4}, and then deleting the last coordinate. 

Thus, we have only to check that the   Lie subalgebra $(V')^\sigma[G^\times]$ is isomorphic to $\f{so}_7(\bb F)$. 
Indeed, if we identify  the vector subspace of $\bb F^8$ spanned by  $\{e_i:i=1,\dots,7\}$ with $\bb F^7$,   
choosing a scalar product $\ang{\cdot,\cdot}$ which makes such basis to be orthonormal,
%again   an orthonormal basis for a scalar product $\ang{\cdot,\cdot}$, 
and we grade $\bb F^7$ by assigning again $\deg(e_i)=g_i\in \bb Z_2^3$ for all $i\in I$,
the related orthogonal Lie algebra  is $\bb Z_2^3$-graded too, and 
$$
\mathcal{B}'_i=\{\varphi_{{i+1}, {i+3}}, \varphi_{{i+2}, {i+6}}, \varphi_{{i+4}, {i+5}} \}\subset  \f{so}_7(\bb F)_{g_i}
$$ 
is a basis of the homogeneous component   of degree ${g_i}$. Moreover,  this basis $\mathcal{B}'_i$ consists of the first 3 vectors of the basis $\mathcal{B}_i$ of $\f{so}_8(\bb F)_{g_i}$.  
\end{proof}

More details on this $\bb Z_2^3$-grading on $\f{so}_7(\bb F)$ are given in \cite[Lemma~2.2]{draper2024gradedb3d4}, so we
will  use the same notation as there, $\Gamma_{\f b_3}$. That work uses the complex field to obtain a complete classification of graded contractions, but   the concrete field was not relevant for describing the main properties of the grading, which remain valid in our setting.

\begin{remark}
As we know, if $\bb F$ is the real field, $\f{so}_8(\bb R)$ and $\f{so}_7(\bb R)$ are compact Lie algebras (negative definite Killing form), so   this will also be the case for
both $V^\sigma[G^\times]$ in Proposition~\ref{pr_d4} and $W^\sigma[G^\times]$ in Corollary~\ref{cor_casob3}. (The Killing form is discussed  below in Section~\ref{se_kappa}.) The use of the group $G=\bb Z_2^3$ is not a coincidence, there does not exist a $H$-grading on any  compact Lie algebra for $H\ne \bb Z_2^n$ (see, for instance, \cite[Proposition~1]{calderondrapermartin2010reales}). Moreover, our algebras are strongly related with the octonion division algebra,   which is, in turn, a twisted group algebra over $\bb Z_2^3$ (\cite{albuquerque1999quiasialgebra}).
\end{remark}

Recall that the only 
example of generalized group  algebra that has already been studied 
   is that one of $\f g_2$ in \cite{draper2024twistedg2} (real and complex field). As this Lie algebra lives inside the orthogonal algebra  $\f{so}_7(\bb F)$, and the related $G$-gradings are compatible, then the
Lie algebra $\f g_2$ should live as a subalgebra of that one in Corollary~\ref{cor_casob3}. It is convenient to locate this subalgebra.

\begin{lemma}\label{re_encajecong2}
%Take $S=\mathrm{span}\langle\{ (0,1,-1),(2,-1,-1)\}\rangle$, vector subspace of $V=\bb F^3$. 
Take the vector subspace $S=\{(a_1,a_2,a_3):a_1+a_2+a_3=0\}$ of $W=\bb F^3$. 
The map
 $\sigma$   in Corollary~\ref{cor_casob3} satisfies  $\sigma_{g,h}(S,S)\subset S$ for all $g,h\in G$, so that $ S^\sigma[G]$ is also a generalized group  algebra.
\end{lemma}

\begin{proof}
Denote by $s_1=(0,1,-1)$ and $s_2=(2,-1,-1)$, a set of generators of $S$. A simple substitution in \eqref{eq_hola2} gives
$$
\begin{array}{llllll}
\hspace{-3pt}\sigma_{i,i+1}\colon&(s_1,s_1)\mapsto \frac12(s_1-s_2)&
\hspace{-5pt}\sigma_{i,i+2}\colon&(s_1,s_1)\mapsto \frac12(s_1-s_2)&
\hspace{-5pt}\sigma_{i,i+4}\colon&(s_1,s_1)\mapsto \frac12(s_1+s_2)\\
&(s_1,s_2)\mapsto \frac32(s_1-s_2)&&(s_1,s_2)\mapsto \frac{-1}2(3s_1+s_2)&&(s_1,s_2)\mapsto \frac{-1}2(3s_1-s_2)\\
&(s_2,s_1)\mapsto \frac{-1}{2}(3s_1+s_2)&&(s_2,s_1)\mapsto \frac{1}{2}(3s_1+s_2)&&(s_2,s_1)\mapsto \frac{-3}{2} (s_1+s_2)\\
&(s_2,s_2)\mapsto \frac12(3s_1+s_2)&&(s_2,s_2)\mapsto \frac{-1}2(9s_1-s_2)&&(s_2,s_2)\mapsto \frac{-1}2(3s_1-s_2).
\end{array}
$$
Also take into account that $\sigma_{g,h}(r,r')=-\sigma_{h,g}(r',r)$ to finish the proof. 
\end{proof}

\begin{prop}\label{pr_casog2}
The generalized group  algebra $ S^\sigma[G]$ as in Lemma~\ref{re_encajecong2}, is sum of a two dimensional center, $Sg_0$ and $\f g_2\equiv S^\sigma[G^\times]$, a simple subalgebra of type $G_2$.
\end{prop}

\begin{proof}  
Note that $s_1$ and $s_2$ in the proof of the above lemma can be naturally identified with the derivations of the octonion algebra $2E_{i}^{\ell_i,i+2}$ and $2F_{i}^{\ell_i}$ from \cite[Eq.~(14)]{draper2024gradedg2}, for $\ell_i=\{i,i+1,i+3\}$,
taking into account that $2E_{i}^{\ell_i,i+2}=  \varphi_{{i+2}, {i+6}}- \varphi_{{i+4}, {i+5}}$ and $2F_{i}^{\ell_i}=2\varphi_{{i+1}, {i+3}}- \varphi_{{i+2}, {i+6}}- \varphi_{{i+4}, {i+5}}$.
\end{proof}
From now on, we will denote by $\Gamma_{\f g_2}$ the grading on the generalized Lie algebra $\f g_2:= S^\sigma[G^\times]$.

In the above  proof we have used our previous knowledge of octonions and the concrete derivations of octonions taken for instance from \cite{draper2024gradedg2}. But
 Remark~\ref{re_claroqueesg2} below provides an alternative proof to this proposition. This means that we have a completely independent construction of $\f g_2$ without using derivations of octonions (not even for the proof), whose bracket is, besides, very easy to handle. 
 
 Our proof that   $\f g_2$ is a Lie algebra over $\bb Z_2^3$   is also  independent  of   \cite{draper2024twistedg2}, although an explicit correspondence can be given by  $(a_1s_1+a_2s_2)g_i\mapsto 2(a_1,a_2)g_i$.  In any case, seeing $\f g_2$ as a subalgebra of  $ W^\sigma[G]$ in such a natural way with $S=\langle(1,1,1)\rangle^\perp$ 
 has made the twist have a much more friendly expression than that of \cite[Theorem~1]{draper2024twistedg2}.

\subsection{Lie algebras over $G$ and the Killing form}\label{se_kappa}  

We will show how well adapted this structure is to the framework of Lie theory, 
providing a completely natural description of the Killing form in terms of the generalized Lie algebra.
In particular, this facilitates the choice of orthonormal bases.

\begin{prop}
Let $\kappa\colon \la\times\la\to\la$ denote the Killing form of a Lie algebra $\la$, that is, $\kappa(x,y)=\mathrm{tr}(\ad(x)\ad(y))$, for $\ad$ the adjoint operator.
As usual consider the group $G=\bb Z_2^3$.
\begin{enumerate}
\item If $ \la=V^\sigma[G]$ is the Lie algebra in Proposition~\ref{pr_d4}, for any $i,j\in I$, $r,s\in V=\bb F^4$,
 $$\kappa(rg_i,sg_j)=-12\langle r,s\rangle\delta_{ij}.$$ 
%Let $V=\bb F^4$ and $G=\bb Z_2^3$. Then 
\item If $ \la=W^\sigma[G]$ is the Lie algebra in Corollary~\ref{cor_casob3}, $i,j\in I$, $r,s\in W=\bb F^3$, then 
\vspace{-2pt}
$$\kappa(rg_i,sg_j)=-10\langle r,s\rangle\delta_{ij}.$$
\item If $ \la=S^\sigma[G]$ is the Lie algebra in Proposition~\ref{pr_casog2}, $i,j\in I$, $r,s\in S= \langle(1,1,1)\rangle^\perp\le\bb F^3$, then
\vspace{-6pt}
$$
\kappa(rg_i,sg_j)=-8\langle r,s\rangle\delta_{ij}.
$$
\end{enumerate}
In the three cases, the neutral component coincides with the radical of $\kappa$.
\end{prop}

(This example makes it very clear that the Killing form of a subalgebra is not the restriction of the Killing form of the algebra to the subalgebra.)

\begin{proof}
Recall that if  $\Gamma: \la = \bigoplus_{g\in G}\la_g$ is a $G$-grading on a Lie algebra, then $\kappa(\la_g,\la_h)=0$ if $g\ne-h$. As now $G=\bb Z_2^3$, any two different homogeneous components are orthogonal for $\kappa$. Then assume that $i=j$, and let us compute the trace of $F=\ad(rg_i)\ad(sg_i)$, for $r,s\in \bb F^4$.
%which is a degree 0 map, since 
Let us denote by  $\{e'_k:k=1,2,3,4\}$   the canonical basis of $V$,
and by $\pi_k\colon Vg_l\to\bb F$ the projection
%projects onto the $k$th component of the vector by 
$\pi_k((a_1,a_2,a_3,a_4)g_l)=a_k$.  (We avoid the name $\pi_k^l$  in order not to complicate the notation.)
Since $F(Vg_l)\subset Vg_l$, 
then
$$\kappa(rg_i,sg_j)=\sum_{l=0\dots7}\sum_{k=1\dots4}\pi_k(F(e_k'g_l)).$$
 For $l\in\{0,i\}$, $F(Vg_l)=0$. 
 %Let us  check that $\sum_{k=1\dots4}\pi_k(F(e_k'g_l))=-2\langle r,s\rangle$ for any $l\ne 0,i$. For instance, f
For $l=i+1$,
$
F(tg_{i+1})=-\sigma_{i,i+4}(\sigma_{i,i+1}(s,t),r)g_{i+1}.
$
So we can compute, for  $r=(a_1,a_2,a_3,a_4)$ and $s= (b_1,b_2,b_3,b_4)$, 
$$
\begin{array}{l}
\pi_1(F(e_1'g_{i+1}))=\pi_3(F(e_3'g_{i+1}))=-a_2b_2-a_3b_3,\quad\\
\pi_2(F(e_2'g_{i+1}))=\pi_4(F(e_4'g_{i+1}))=-a_1b_1-a_4b_4,
\end{array}
$$
which gives $\sum_{k=1\dots4}\pi_k(F(e_k'g_{i+1}))=-2\langle r,s\rangle$. Similarly we proceed with any $l\ne 0,i$ to obtain 
$\sum_{k=1\dots4}\pi_k(F(e_k'g_l))=-2\langle r,s\rangle$. Thus $\kappa(rg_i,sg_i)=-12\langle r,s\rangle$, finishing the case $[\la,\la]\cong\f{so}_8(\bb F)$.

For the second case, we have to be careful, because the Killing form does not restrict well to subalgebras. 
Besides, although we can argue quite similarly,   now not all the partial sums are equal. What happens is,
for $F=\ad(rg_i)\ad(sg_i)$,  %$r,s\in \bb F^3$, 
$r=(a_1,a_2,a_3)$, $s=(b_1,b_2,b_3)$,
$$
\begin{array}{l}
\sum_{k=1\dots4}\pi_k(F(e_k'g_{i+1}))=\sum_{k=1\dots4}\pi_k(F(e_k'g_{i+3}))=-a_1b_1-2a_2b_2-2a_3b_3,\\
\sum_{k=1\dots4}\pi_k(F(e_k'g_{i+2}))=\sum_{k=1\dots4}\pi_k(F(e_k'g_{i+6}))=-2a_1b_1-a_2b_2-2a_3b_3,\\
\sum_{k=1\dots4}\pi_k(F(e_k'g_{i+4}))=\sum_{k=1\dots4}\pi_k(F(e_k'g_{i+5}))=-2a_1b_1-2a_2b_2-a_3b_3,
\end{array}
$$
which gives trace of $F$ equal to $-10\langle r,s\rangle$.

Finally, for the algebra $  S^\sigma[G]$, it is enough to check that $\kappa(s_1g_i,s_1g_i)=-16$, $\kappa(s_1g_i,s_2g_i)=0$ and $\kappa(s_1g_i,s_2g_i)=-48$, as then the bilinearity gives the result. %<s1,s1>=2,<s1,s2>=0,<s2,s2>=6
For instance,  the map $F=\ad(s_1g_i)\ad(s_1g_i)$ can be tediously computed:
$$
\begin{array}{lll}
s_1g_{i+1}\mapsto (-s_1-s_2)g_{i+1},\qquad&s_1g_{i+2}\mapsto -s_1g_{i+2},\ \qquad&s_1g_{i+4}\mapsto -s_1g_{i+4},\\
s_2g_{i+1}\mapsto (-3s_1-3s_2)g_{i+1},&s_2g_{i+2}\mapsto -s_2g_{i+2},&s_2g_{i+4}\mapsto -s_2g_{i+4},\\
s_1g_{i+3}\mapsto (-s_1+s_2)g_{i+3},&s_1g_{i+6}\mapsto -s_1g_{i+6},&s_1g_{i+5}\mapsto -s_1g_{i+5},\\
s_2g_{i+3}\mapsto (3s_1-3s_2)g_{i+3},&s_2g_{i+6}\mapsto -s_2g_{i+6},&s_2g_{i+5}\mapsto -s_2g_{i+5},\\
\end{array}
$$ 
which gives $\mathrm{tr}(F)=-1-3-1-3-1-1-1-1-1-1-1-1=-16$. Proceed similarly with the other two cases.
%% $\mathrm{tr}\ad(s_2g_i)\ad(s_2g_i)=-1(4)-3(4)-7(2)-9(2)=-48$ por si alguien pide completar
\end{proof}

\begin{remark}\label{re_claroqueesg2}
The previous proposition implies the semisimplicity of the derived algebra of $  S^\sigma[G]$, since the Killing form is nondegenerate (in fact, it is negative definite in case $\bb F=\bb R$). 
It is not very difficult to conclude that it has type $G_2$, without any other consideration on derivations of octonions. Indeed, the only other semisimple Lie algebra of dimension 14 has type $A_2\oplus2A_1$, by dimension count (the only simple Lie algebras with dimensions less than 14 have dimension 3, 8 and 10), which cannot be contained in an algebra of type $B_3$, arguing about the rank.  
To summarize, we have an alternative proof of Proposition~\ref{pr_casog2} which does not need to use \cite{draper2024gradedg2}.
\end{remark}

%%%%%%%%%%%%%%%%%%%%%%%%%%%%%%%%%%%%

\subsection{Generalized group algebras and irreducible representations}\label{se_rep}

It is well-known, if $\bb F$ is  an algebraically closed field  of characteristic zero, that $\f{so}_8(\bb F)$ has 4 basic irreducible representations, and any other irreducible representation lives as a submodule of the tensor product of copies of the basic ones. 
This makes it important to describe these 4 representations in an easy way. (Material about representations on Lie algebras can be consulted, for instance, in \cite{humphreys1972intro_to_lie_and_rep}.)
 Besides the adjoint module, the other 3 basic modules have all dimension 8 and are   the natural one, and the two half-spin modules. Our next purpose  is to describe these modules from the point of view of the generalized Lie algebra. (The adjoint module is of course not necessary.) They are surprisingly well adapted to our description in terms of the group $\bb Z_2^3$.
 
 \begin{prop}\label{pr_espines}
The 3 non-equivalent irreducible representations of dimension 8 of the generalized Lie algebra $[\la,\la]=V^\sigma[G^\times]$ are given by
$\rho_k\colon V^\sigma[G^\times]\to\f{gl}_8(\bb F)$, $k=1,2,3$, for
$$
\begin{array}{l}
\rho_1((a,b,c,d)g_i)=\delta(i)_{(a,b,c,d)}\\
\rho_2((a,b,c,d)g_i)=\delta(i)_{(a,b,c,d)-\frac{a+b+c-d}2(1,1,1,-1)}\\
\rho_3((a,b,c,d)g_i)=\delta(i)_{(a,b,c,d)-\frac{a+b+c+d}2(1,1,1,1)}
\end{array}
$$
where   $\delta(i)_{(a,b,c,d)}\colon \bb F^8\to  \bb F^8$ denotes the linear map:
$$
\begin{array}{llll}
e_{i+1}\mapsto ae_{i+3},\qquad &e_{i+2}\mapsto be_{i+6},\qquad &e_{i+4}\mapsto ce_{i+5},\qquad &e_{0}\mapsto de_{i},\\
e_{i+3}\mapsto -ae_{i+1},\qquad &e_{i+6}\mapsto -be_{i+2},\qquad &e_{i+5}\mapsto -ce_{i+4},\qquad &
e_{i}\mapsto -de_{0}.
\end{array}
$$
 \end{prop}
 
 \begin{proof}
The   natural representation $\rho_1$ is clear, due to 
$$
\varphi_{{i+1}, {i+3}}=\delta(i)_{(1,0,0,0)}, \quad
\varphi_{{i+2}, {i+6}}=\delta(i)_{(0,1,0,0)},  \quad
\varphi_{{i+4}, {i+5}}=\delta(i)_{(0,0,1,0)},  \quad
\varphi_{0, i}=\delta(i)_{(0,0,0,1)}.
$$

For describing the two other representations, it is convenient to think of $\bb F^8$ as an algebra $\mathcal O=\bb F^8$ with the product where $e_0=1$ is the unity, $e_i^2=-1$  for all $i\in I$ and
\begin{equation}\label{eq_octo}
e_ie_{i+1}=-e_{i+1}e_{i}=e_{i+3},
\end{equation}
and all the 
cyclic permutations of this identity hold, always taking the sum modulo 7. (This is the Cayley algebra, or octonion algebra, if the field has characteristic different from 2 and 3.)
According to the   principle of local triality (see, for instance, \cite[Theorem~3.31]{schafer2017intro_to_nonassociative}),   for every  $U\in\f{so}_8(\bb F)$
there are unique $U',U''\in \f{so}_8(\bb F)$ satisfying 
\begin{equation*}\label{eq_trialdad1}
U(xy)=U'(x)y+xU''(y) \qquad \textrm{ for all }x,y\in \mathcal O.
\end{equation*}
Precisely the representations $\rho_2,\rho_3\colon \f{so}_8(\bb F)\to \f{gl}_8(\bb F)$ come from assigning $\rho_2(U)=U'$ and $\rho_3(U)=U''$. Thus, it is convenient to recall the proof of this principle,
based on the well-known fact  that
\begin{equation}\label{eq_descomp}
\f{so}_8(\bb F)=\f{der}(\mathcal O)\oplus L_{\mathcal O_0}\oplus R_{\mathcal O_0},
\end{equation}
 for $\f{der}(\mathcal O)$ the derivation algebra, $\mathcal O_0=\langle\{e_i:i\in I\}\rangle$,
 and where $L_x,R_x\colon\mathcal O\to\mathcal O$ denote the left and right multiplication operators  $L_x(y)=xy$ and $R_x(y)=yx.$ 
Now, if $d\in\f{der}(\mathcal O)$, that is, $d(xy)=d(x)y+xd(y)$, this means that $d=d'=d''$. 
On the other hand, the alternativity of the algebra $\mathcal O$ (that is, $x^2y=x(xy)$ and $yx^2=(yx)x$) gives, for any $x\in\mathcal O_0$,
\begin{equation}\label{eq_trialdad2}
(L_x)'=R_x+L_x,\qquad (L_x)''=-L_x,\qquad
(R_x)'=-R_x , \qquad(R_x)''=R_x+L_x. 
\end{equation}
In particular all this can be applied to our skew-symmetric maps $\varphi_{x,y}$. This requires of
computing  its decomposition according to Eq.~\eqref{eq_descomp}. Once we check  
that, for any $x,y\in\mathcal O_0$, 
$$
 \varphi_{x,y}+\frac1{12}\big(R_{[x,y]}-L_{[x,y]} \big)=\frac16\big([L_x,L_y]+[L_x,R_y]+[R_x,R_y]\big)
\in\f{der}(\mathcal O)
$$
(the derivations on the alternative algebra $\mathcal O$ are described in  \cite{schafer2017intro_to_nonassociative}),   and 
 $$
\varphi_{1,x}=\frac12(R_x+L_x),
$$
then Eq.~\eqref{eq_trialdad2} gives immediately  
\begin{equation}\label{eq_accionessemiespin}
\begin{array}{ll}
 (\varphi_{x,y})'=\varphi_{x,y}+\frac14R_{[x,y]},\qquad\qquad& (\varphi_{1,x})'=\frac12L_x=\varphi_{1,x}-\frac12R_x,\\
 (\varphi_{x,y})''=\varphi_{x,y}-\frac14L_{[x,y]},& (\varphi_{1,x})''=\frac12R_x=\varphi_{1,x}-\frac12L_x.\\
\end{array}
\end{equation}
According to Eq.~\eqref{eq_octo}, 
$
{[}e_{i+1},e_{i+3}]=[e_{i+2},e_{i+6}]=[e_{i+4},e_{i+5}]=2e_{i},
$
and
$$
L_{e_i}=\delta(i)_{(1,1,1,1)}, \qquad R_{e_i}=\delta(i)_{(-1,-1,-1,1)}. 
$$
The only thing left to do is to put this together with Eq.~\eqref{eq_accionessemiespin} to get
$\rho_2((a,b,c,d)g_i)=\delta(i)_{(a,b,c,d)}+\frac{a+b+c-d}2\delta(i)_{(-1,-1,-1,1)}$
and 
$\rho_3((a,b,c,d)g_i)=\delta(i)_{(a,b,c,d)}-\frac{a+b+c+d}2\delta(i)_{(1,1,1,1)}$. %%realmente todo no, basta juntar las 2 ultimas
More familiar expressions follows immediately,
$$
\begin{array}{l}
\rho_2((a,b,c,d)g_i)=  \delta(i)_{\big( \frac{a-b-c+d}2,\frac{-a+b-c+d}2,\frac{-a-b+c+d}2,\frac{a+b+c+d}2 \big)},\\
\rho_3((a,b,c,d)g_i)=  \delta(i)_{\big( \frac{a-b-c-d}2,\frac{-a+b-c-d}2,\frac{-a-b+c-d}2,\frac{-a-b-c+d}2 \big) } .
\end{array}
$$
 \end{proof}
 
 In particular, all the (finite) irreducible modules for $\f{so}_8(\bb F)$ admit $G$-gradings compatible with the $G$-grading on the Lie algebra. This is well-known from \cite{EK_Gradedmodules_compatibles}, 
 but Proposition~\ref{pr_espines} takes advantage of it to provide concrete -and again, very easy- expressions of these actions.
 
 \begin{remark}
 Note that, for $u_1=(1,1,1,-1)$ and $u_2=(1,1,1,1)$, we can write $\rho_1(vg_i)=\delta(i)_{v}$, $\rho_2(vg_i)=\delta(i)_{s_{u_1}(v)}$, 
 and $\rho_3(vg_i)=\delta(i)_{s_{u_2}(v)}$, for any $v\in V$, where $s_u(v)=v-2\frac{\langle {v,u}\rangle }{\langle {u,u}\rangle }u$ denotes the (order 2) reflection
  through the hyperplane $u^\perp$.
  The \emph{triality automorphism} is  an order 3 automorphism of $\f{so}_8(\bb F)$ which permutes the three inequivalent representations. It can be obtained by composing these order 2 automorphisms of $V^\sigma[G]$: 
  $$
  vg_i\mapsto s_{u_1}(v)g_i,\qquad vg_i\mapsto s_{u_2}(v)g_i,
  $$
  which translate the automorphisms of $\f{so}_8(\bb F)$ given by
  $$
  U\mapsto U',\qquad U\mapsto U''.
  $$
   (Note that, $\sigma_{ij}(s_{u_1}(v),s_{u_1}(v'))=s_{u_1}(\sigma_{ij}(v,v'))$ for all $i,j\in I$, $v,v'\in V$, which is the condition for $vg_i\mapsto s_{u_1}(v)g_i$ to define an automorphism. Similarly occurs for $u_2$, but it is not a general fact for any $u$.)
  Hence,  the triality automorphism 
  $\theta\colon \f{so}_8(\bb F)\to \f{so}_8(\bb F)$, $\theta(U')=U''$ for any $U\in \f{so}_8(\bb F)$,
  can be described in terms of the Lie algebra over $\bb Z_2^3$ as $\theta(vg_i)=s_{u_2}s_{u_1}(v) g_i$. 
  In other words, in column notation, 
  $$
  \theta(vg_i)=
  \left(\frac12\left(
\tiny\begin{array}{cccc}
 1 & -1 & -1 & 1 \\
 -1 & 1 & -1 & 1 \\
 -1 & -1 & 1 & 1 \\
 -1 & -1 & -1 & -1 \\
\end{array}
\right)v\right)g_i.
$$
Remarkably, the fixed subalgebra   $\mathrm{Fix}(\theta)=\{(a,b,c,d)g_i:i\in I,a+b+c=d=0\}=S[G^\times]$ coincides, as expected, with $\f{g}_2$.
 \end{remark}
 \smallskip

The case we are interested in highlighting is $\f g_2$, 
 with 2 basic irreducible representations: the adjoint representation and the natural representation, of dimension 7. 
  It can be extracted from the natural representation of $\f{so}_7(\bb F)\equiv W^\sigma[G^\times]$, which is the own $\bb F^7=\langle \{e_i:i\in I\}\rangle $. Using the notations in Proposition~\ref{pr_espines}, the irreducible representations of dimension 7 for $W^\sigma[G^\times]$ and $S^\sigma[G^\times]$ are given, respectively, by
$$
\begin{array}{ll}
\rho_{\f b_3}\colon W^\sigma[G^\times]\to\f{gl}_7(\bb F), &\rho_{\f b_3}((a,b,c)g_i)=\delta(i)_{(a,b,c,0)}\vert_{\bb F^7}, \\
\rho_{\f g_2}\colon S^\sigma[G^\times]\to\f{gl}_7(\bb F), &\rho_{\f g_2}((a,b,c)g_i)=\delta(i)_{(a,b,c,0)}\vert_{\bb F^7}.
\end{array}
$$
 The action is easily written in terms of the basis $B_i=\{e_{i},e_{i+1},e_{i+2}, e_{i+3},e_{i+4},e_{i+5},e_{i+6}\}$ of $\bb F$.
 The coordinates of a vector in this basis will be denoted with the subindex $_{B_i}$.
\begin{cor}
The   irreducible representation  
 $\rho_{\f g_2}\colon S^\sigma[G^\times]\to\f{gl}_7(\bb F)$ can be described by
 $$
 \begin{array}{l}
 \rho_{\f g_2}(s_1g_i): 
 (\alpha_{0},\alpha_{1},\alpha_{2},\alpha_{3},\alpha_{4},\alpha_{5},\alpha_{6})_{B_i}\mapsto
(0,0,-\alpha_{6},0,\alpha_{5},-\alpha_{4},\alpha_{2})_{B_i},
 \\
 \rho_{\f g_2}(s_2g_i): 
 (\alpha_{0},\alpha_{1},\alpha_{2},\alpha_{3},\alpha_{4},\alpha_{5},\alpha_{6})_{B_i}\mapsto
(0,-2\alpha_{3},\alpha_{6},2\alpha_{1},\alpha_{5},-\alpha_{4},-\alpha_{2})_{B_i}.
 \end{array}
 $$
 \end{cor}
 This avoids using the octonion algebra, and especially, it avoids using derivations of the octonion algebra, 
 which are obviously   painful. (Note that changing from the basis $B_i$ to $B_j$ only involves shifting the coordinates in cycles.)
 
\section{generalized group  algebras coming from graded contractions}\label{se_3}

More examples of non-necessarily reductive Lie algebras which are generalized group  algebras can be obtained with a tool proposed by physicists, that one of graded contractions.

\subsection{Preliminaries on graded contractions}\label{subsec_prelim_gr_cont}

This section is mainly extracted from \cite[Sections 2 and 3]{draper2024gradedg2}, although there the chosen field is $\bb C$. Note that the proofs can be adapted without significative changes.
Let $G$ be an arbitrary abelian group. 

\begin{define}
Let $\Gamma: \la = \bigoplus_{g\in G} \la_g$ be a $G $-grading on a   Lie algebra $\la$ over $\bb F$.
\begin{itemize}
\item
A \emph{graded contraction} of $\Gamma$ is a map $\varepsilon\colon G\times G\to \bb{F}$ such that the vector space $\la$ endowed with the product $[x, y]^\varepsilon := \varepsilon(g, h)[x, y]$, for $x\in \la_g,  y\in \la_h$, $g,h\in G$, is a Lie algebra. We write $\la^\ep$ to refer to $(\la, [\cdot, \cdot]^\epsilon)$,
which is $G$-graded too with $(\la^\ep)_g=\la_g$. 

\item
We will say that two graded contractions $\ep$ and $\ep'$ of $\Gamma$
are  \emph{equivalent}, written $\ep\sim\ep'$, if $\la^{\ep}$ and $\la^{\ep'}$ are isomorphic as   graded algebras, i.e., 
there is an isomorphism of  Lie algebras  $f\colon\la^\ep\to\la^{\ep'}$   such that for any $g\in G$ there is $h \in G$ with $f(\la_g) = \la_h$.
\end{itemize}
\end{define}

\begin{remark}\label{re_general} (\cite[Remark~2.9]{draper2024gradedg2})
 If $\Gamma$ is a grading on a Lie algebra $\la$,  
 an arbitrary map $\ep\colon G\times G\to \bb{F}$ is a graded contraction of $\Gamma$  if and only if 
 \begin{enumerate}
\item[(a1)] $\big(\ep(g, h) - \ep(h, g)\big)[x, y] = 0,$
\item[(a2)] $\big(\ep(g, h, k) - \ep(k, g, h)\big)[x, [y, z]]
+ \big(\ep(h, k, g) - \ep(k, g, h)\big)[y, [z, x]] = 0$,
\end{enumerate}
 for all $g, h, k \in G$ and any choice of homogeneous elements $x\in \la_g, \, y\in \la_h, \, z\in\la_k$. 
 Here $\ep\colon G \times G \times G \to \bb{F}$ denotes the ternary map defined as
$\ep(g, h, k) := \ep(g, h + k)\ep(h, k)$.
 \end{remark}
 
  These conditions are in general strongly dependent on   the considered grading $\Gamma$ on $\la$. 
  But the $ \bb Z_2^3$-gradings $\Gamma_{\f d_4}$, $\Gamma_{\f b_3}$ and $\Gamma_{\f g_2}$ 
   %have some properties which help to a common treatment:  
  have some properties (\cite[Lemma~2.2]{draper2024gradedb3d4})
  that make it possible to give them a common treatment:
      \begin{enumerate}
     \item[(i)] $\la_e=0$;
     \item[(ii)]  $[\la_g,\la_h]=\la_{g+h}$ for all $g\ne h\in G^\times$;
     \item[(iii)]  If  $\langle g, h, k \rangle=G$, then there exist $x\in\la_g$, $y\in\la_h$ and $z\in\la_k$ such that 
     the set $\{[x,[y,z]],[y,[z,x]]\}$ is   linearly independent.
     \end{enumerate}

         This permits easily to prove (arguments as in \cite[Lemma~3.2]{draper2024gradedg2}) that 
         \begin{lemma}
         For any graded contraction $\ep$   of   $\Gamma\in\{\Gamma_{\f d_4},\Gamma_{\f b_3},\Gamma_{\f g_2}\}$,  %a good grading!!
 there exists another   graded contraction $\ep'$ of $\Gamma$   equivalent to $\ep$ satisfying $\ep'(g, g) = \ep'(e, g) = \ep'(g, e) = 0$.
Any map $\ep'\colon G\times G \to \bb{F}$ satisfying this condition  will be called \emph{admissible}.
\end{lemma}

Conditions in Remark~\ref{re_general}  can be weakened for admissible maps: 
	
\begin{lemma}\label{ref_existeadmisible}
An admissible map $\ep\colon  G \times G \to \bb{F}$ is a graded contraction of 
$\,\Gamma\in\{\Gamma_{\f d_4},\Gamma_{\f b_3},\Gamma_{\f g_2}\}$
%either  $\Gamma_{\f d_4}$ and $\Gamma_{\f b_3}$
 if and only if the following conditions hold for all $g, h, k \in G$:
\begin{enumerate}
\item[\rm (a1)'] $\ep(g, h) = \ep(h, g)$,  
\item[\rm (a2)'] $\ep(g, h, k) = \ep(k, g, h)$, provided that $G = \langle g, h, k \rangle$. 
\end{enumerate}
\end{lemma}

Ultimately, this enables to find, in \cite{draper2024gradedb3d4},  all the admissible graded contractions of $\Gamma_{\f d_4}$ and $\Gamma_{\f b_3}$ up to equivalence for $\bb F=\bb C$, adapting the results on \cite{draper2024gradedg2} about   the $\bb Z_2^3$-grading  on the simple Lie algebra $\f g_2$  obtained as derivations of the octonion algebra (our $\Gamma_{\f g_2}$). 
Some of the results are valid independently of the field, 
but not all, as shown, for instance, in \cite[Proposition~4.1]{draper2024gradedb3d4}, which contains some comments on the real field.

%%%%%%%%%%%%%%%%%%%%%%%%%%%%%%%%%%%%%%%%%%%%%%% 
 \subsection{Graded contractions of $\Gamma_{\f d_4}$, $\Gamma_{\f b_3}$ and $\Gamma_{\f g_2}$}\label{subsec_clasificacion}
 
 Let us recall the above mentioned classification. The key concept is that of support: 
 \begin{define} \label{def_sup}
 Take $X:=\{\{i,j\}:1\le i<j \le7\}$.
 For any
 admissible graded contraction $\ep\colon  G \times G \to \bb{F}$, 
  its \emph{support} is defined by
 $\supp\ep:=\{\{i,j\}\in X:\ep(g_i,g_j)\ne0\}.$
 \end{define}
 
 The support can not be an arbitrary subset of $X$.
 
 \begin{define} \label{def_nice}
If $\langle g_i,g_j,g_k\rangle=G$, that is, for $i, j, k \in I$ distinct with   $k\ne i*j$, consider the cardinal 6 set
 \[
P_{\{i,j,k\}} := \{\{i, j\}, \{j,k\},  \{k,i\}, \{i, j\ast k\}, \{j, k\ast i\}, \{k, i\ast j\}\}\subset X.
\]
A subset $T\subset X$ is called \emph{nice} if,
whenever $i, j, k \in I$ distinct with   $k\ne i*j$,
 $\{i, j\},   \{i\ast j, k\} \in T$, we have $P_{\{i,j,k\}} \subset T$.

\end{define}
 As proved in \cite[Proposition~3.10]{draper2024gradedg2}, the support of any admissible graded contraction is a nice set. And conversely, if $T\subset X$ is a nice set, then $\ep^T\colon G\times G\to\bb F$ is always an admissible graded contraction, for
\begin{equation}\label{eq_ET}
   \ep^T(g_i,g_j)=
   \begin{cases}
   1\qquad\textrm{if $\{i,j\}\in T$,}\\
   0\qquad\textrm{otherwise. }
   \end{cases}  
\end{equation}

Observe that we are not specifying which of the three gradings, $\Gamma_{\f d_4}$, $\Gamma_{\f b_3}$ or $\Gamma_{\f g_2}$, we are referring to: this is due to the surprising result that an admissible map $\ep\colon  G \times G \to \bb{F}$ is a graded contraction of one of such gradings if and only if it is a graded contraction of the other two. 
 However, do not forget that the Lie algebras $\la^\ep$ obtained depend on $\la$, not only on 
 the map  $\ep$, so that with $\la^\ep$ we are referring to several Lie algebras   even of different dimension.

As we are interested in getting non-isomorphic  Lie algebras,  
we have to study the equivalence of graded contractions.
We say that  two nice sets $T$ and $T'$ are \emph{collinear} if there is   a bijection $\mu\colon I\to I$ such that $\mu(i)*\mu(j)=\mu(i*j)$ for all $i\ne j$ and $\{\{\mu(i),\mu(j)\}:\{i,j\}\in T\}= T'$. (The term \emph{collineation} 
for $\mu$
comes from preserving the lines of the Fano plane $PG(2,2)$.) The properties of the three considered gradings, concretely the fact that they share the group of symmetries of the grading, the so-called \emph{Weyl group of the grading}, allowed us to prove that if $T$ and $T'$ were collinear, then $\ep^{T}$ and $\ep^{T'}$ would be equivalent (regardless of whether the algebra under consideration was  $\f g_2^{\bb C}$, $\f{so}_7(\bb C)$ or $\f{so}_8(\bb C)$). Surprisingly, the converse,   which seemed to be true, is not true, but  is \lq nearly\rq\  true:
%is not true \lq by a little\rq: 
there are 24 equivalence classes of non-collinear nice sets, and the corresponding  graded contractions by \eqref{eq_ET} are all not equivalent except for only one case. Best of all,  this follows being  true by replacing the complex field by any other field of characteristic different from zero (even this restriction could be weakened).
According to the classification of the nice sets up to collineations, a tedious purely combinatorial task  completed in \cite[Theorem~3.27]{draper2024gradedg2}, 
we can choose the following   representatives of the   classes of the nice sets up to collineations:
\begin{itemize}
    \item $T_1:=\emptyset$;
    \item $T_2:=\{ \{1,2\}\}$;
    \item $T_3:=\{ \{1,2\},\{1,3\}\}$;
    \item $T_4:=\{\{ 1,2\},\{ 1,4\} \}$;
    \item $T_5:=\{ \{1,2 \},\{ 5,7\}\}$;
    \item $T_6:=\{ \{ 1,2\},\{1,4 \},\{2,4 \}\}$;
    \item $T_7:=\{ \{2,4 \},\{ 3,7\},\{ 5,6\}\}$;
    \item $T_8:=\{\{1,2\},\{1,3\},\{1,6\}\}$;
    \item $T_9:=\{ \{ 1,2\},\{ 1,3\},\{ 1,4\}\}$;
    \item $T_{10}:=\{ \{ 1,2\},\{ 1,3\},\{ 1,5\}\}$;
    \item $T_{11}:=\{ \{ 1,2\},\{ 1,7\},\{ 2,7\}\}$;
    \item $T_{12}:=\{ \{ 1,2\},\{ 1,7\},\{ 5,7\}\}$;
    \item $T_{13}:=\{ \{ 1,2\},\{ 1,3\},\{1,4\},\{ 1,6\}\}$;
    \item $T_{14}:=\{\{1,2\},\{1,3\},\{1,4\},\{1,7\}\}$;
    \item $T_{15}:=\{ \{ 1,2\},\{ 1,5\},\{ 1,7\},\{2,7 \}\}$;
    \item $T_{16}:=\{ \{ 1,2\},\{ 1,7\},\{ 2,5\},\{ 5,7\}\}$;
    \item $T_{17}:=\{\{1,2\},\{1,3\},\{1,4\},\{1,6\},\{1,7\}\} $;
    \item $T_{18}:=\{ \{ 1,2\},\{ 1,5\},\{ 1,7\},\{ 2,5\},\{ 2,7\} \}$;
    \item $T_{19}:=\{ \{ 3,5\},\{ 3,6\},\{3,7 \},\{ 5,6\},\{ 5,7\},\{ 6,7\} \} $;
    \item $T_{20}:=\{\{1,2\},\{1,3\},\{1,4\},\{1,5\},\{1,6\},\{1,7\}\} $;
    \item $T_{21}:=\{ \{1,2 \},\{ 1,3\},\{ 1,5\},\{ 2,3\},\{ 2,7\},\{ 3,4\} \}=P_{\{1,2,3\}}$;
    \item $T_{22}:=\{ \{1,2 \},\{ 1,3\},\{ 1,4\},\{ 1,5\},\{ 1,6\},\{ 1,7\}, \{ 2,3\},\{ 2,7\},\{3,4 \},\{4,7\}\} $;
    \item $T_{23}:=X-T_{19}$;
    \item $T_{24}:=X$.
\end{itemize}
(The elements in the $T_i$'s appear to have been changed from \cite{draper2024gradedg2}, simply because the 
labelling of the elements of $\bb Z_2^3$ in \eqref{eq: elementos de Z_2^3}   is different from that one in \cite{draper2024gradedg2}.)
Furthermore, for any $i\ne j$, $\ep^{T_i}$ is not equivalent to $\ep^{T_j}$ except for the case $\{i,j\}=\{8,10\}$ (\cite[Proposition~4.11]{draper2024gradedg2}).

These are not  the only non-equivalent graded contractions. 
For instance, consider, for any $\lambda,\lambda_1,\lambda_2 \in \mathbb{F}-\{0\}$ %the next graded contractions:
the admissible maps $\eta^\lambda,\mu^\lambda,\beta^{\lambda_1,\lambda_2}\colon  G \times G \to \bb{F}$ given by
\begin{enumerate}
   % \item $\varepsilon^T: X \longrightarrow \mathbb{F}$ with support $T$ and $\varepsilon_{i,j}^T=1$ for any $\{i,j\}\in T$;
    \item $\eta^\lambda_{i,j}=0$ for $\{i,j\}\notin T_{14}$, $\eta^\lambda_{1,2}=\eta^\lambda_{1,3}=\eta^\lambda_{1,4}=1$  and $\eta^\lambda_{1,7}=\lambda$;
    
    \item $\mu^\lambda_{i,j}=0$ for $\{i,j\}\notin T_{17}$, $\mu^\lambda_{1,2}=\mu^\lambda_{1,4}=\mu^\lambda_{1,6}=1$  and $\mu^\lambda_{1,3}= \mu^\lambda_{1,7}=\lambda$;
    
    \item $\beta^{\lambda_1,\lambda_2}_{i,j}=0$ for $\{i,j\}\notin T_{20}$, $\beta^{\lambda_1,\lambda_2}_{1,2}=\beta^{\lambda_1,\lambda_2}_{1,4}=1$, $\beta^{\lambda_1,\lambda_2}_{1,3}=\beta^{\lambda_1,\lambda_2}_{1,7}=\lambda_1$  and    $\beta^{\lambda_1,\lambda_2}_{1,5}=\beta^{\lambda_1,\lambda_2}_{1,6}=\lambda_2$;
\end{enumerate}
where we write   $\eta_{ij}=\eta(g_i,g_j)$ for any admissible map $\eta$.
(With this notation, $\eta^1=\ep^{T_{14}}$, $\mu^1=\ep^{T_{17}}$ and $\beta^{1,1}=\ep^{T_{20}}$.) These are graded contractions too. In fact, for the complex field, they provide the classification of the graded contractions up to equivalence:
\begin{theorem}
    \label{thm: clases de isomorfia de las contraciones graduadas}
   (\cite[Theorem~4.13]{draper2024gradedg2}  and \cite[Theorem 3.24]{draper2024gradedb3d4})
   For $\bb F=\bb C$, representatives of all isomorphism classes up to equivalence of the graded contractions of any $\Gamma\in\{\Gamma_{\mathfrak{d}_4},\Gamma_{\mathfrak{b}_3},\Gamma_{\mathfrak{g}_2}\}$ are just:
    \begin{itemize}
        \item $\varepsilon^{T_i}$ with $i\neq 8,14,17,20$;
        \item $\eta^\lambda$ with $\lambda\in \mathbb{F}-\{0\}$, where $\eta^\lambda \sim \eta^{\lambda'}$ if and only if $\lambda'\in \{\lambda,\lambda^{-1}\}$;
        \item $\mu^\lambda$ with $\lambda\in \mathbb{F}-\{0\}$, where $\mu^\lambda \sim \mu^{\lambda'}$ if and only if $\lambda'\in \{\pm \lambda,\pm \lambda^{-1}\}$;
        \item $\beta^{\lambda_1,\lambda_2}$ with $\lambda_1,\lambda_2\in \mathbb{F}-\{0\}$, where $\beta^{\lambda_1,\lambda_2} \sim \beta^{\lambda_1',\lambda_2'}$ if and only if the set $\{\pm \lambda_1',\pm \lambda_2'\}$ coincides with either $\{\pm \lambda_1,\pm \lambda_2\}$  or $\{\pm \lambda_1^{-1},\pm \lambda_2\lambda_1^{-1}\}$ or $\{\pm \lambda_2^{-1},\pm \lambda_1\lambda_2^{-1}\}$.
    \end{itemize}
\end{theorem}

  In the real case, a lot of work remains to be done to achieve a complete classification: 
  all the above provide non-equivalent graded contractions but the list is far from exhaustive.
Do not forget \cite[Proposition~4.1]{draper2024gradedb3d4} that for the three considered algebras there are admissible graded contractions with  support equal to $X$ which allow to pass from the compact algebra to the split form, obviously not isomorphic.

  %%%%%%%%%%%%%%%%%%%%%%%%%%%%%%%%%
 \subsection{Graded contractions on generalized group  algebras}\label{sec_ultima}
 
 The following observation is trivial but crucial for our purposes.

\begin{lemma}\label{le_crucial}
Let     $\la= V^\sigma[G]$ a Lie  algebra over $G$, for $V$   and $\sigma$ as in Definition~\ref{def_twisted}.
Let $\Gamma_{(V,\sigma,G)}$ the $G$-grading on $\la$ given by $\la_g=Vg$.
For any graded contraction $\varepsilon\colon G\times G\to \bb{F}$ of $\Gamma_{(V,\sigma,G)}$, the algebra $\la^\ep= V^{\ep\sigma}[G]$ is again a Lie  algebra over $G$.
\end{lemma}

\begin{proof}
 Take $\tilde\sigma: G\times G \rightarrow \mathrm{Bil}(V\times V,V),\,(g,h)\mapsto\ep(g,h)\sigma_{g,h}$. Let us check that 
$\la^\ep=V^{\tilde\sigma}[G]$. Indeed, for homogeneous elements in $\la$, $x=rg$ and $y=sh$, $r,s\in V$,
$$
[x,y]^\ep=\ep(g,h)[x,y]=\ep(g,h)\sigma_{g,h}(r,s)(g+h)=\tilde\sigma_{g,h}(r,s)(g+h),
$$
which coincides with the bracket in $V^{\ep\sigma}[G]$, so that $\la^\ep$ is a generalized group algebra.
This finishes the argument, since the algebra  $\la^\ep=(\la,[\,,\,]^\ep)$ is Lie by the own definition of graded contraction. 
\end{proof}

This means that all the Lie algebras obtained by means of a graded contraction of $\Gamma_{\f d_4}$, $\Gamma_{\f b_3}$ and $\Gamma_{\f g_2}$ as in Section~\ref{subsec_clasificacion} are examples of generalized group  algebras. This is important for us, because it provides immediately an important collection of examples of generalized group  algebras, 
showing that the example of $\f g_2$ was not isolated at all.
More details on the properties satisfied by the obtained algebras were exhibited in \cite[Theorem~5.1]{draper2024gradedg2}:
thus  there are generalized group  algebras of   very different nature: reductive, nilpotent, solvable but not nilpotent, and so on.  
Although we are far from a classification of the generalized group  algebras which are Lie algebras, 
we have contributed in our first objective,
to highlight the possible importance of the concept of generalized group  algebras in the Lie theory setting.

Moreover, the concrete expressions of the twists $\sigma$'s   
in Proposition~\ref{pr_d4} and Corollary~\ref{cor_casob3}
can be combined with the graded contractions 
described in Theorem~\ref{thm: clases de isomorfia de las contraciones graduadas}
by means of Lemma~\ref{le_crucial},  thus getting totally precise expressions for the twists related to the new family of 
generalized group  algebras.

\begin{cor}\label{cor_results} 
Let $V=\bb F^4$ and $G=\bb Z_2^3$. Then $ \la=V^{\tilde\sigma}[G]$ is a Lie  algebra over $ \bb Z_2^3$
for any $\tilde \sigma\colon G\times G \rightarrow \mathrm{Bil}(V\times V,V)$ in the next list:
$$
\tilde\sigma\in\{\varepsilon^{T_i}\sigma ,\eta^\lambda\sigma , \mu^\lambda\sigma ,\beta^{\lambda_1,\lambda_2}\sigma:i=1,\dots,24;\,\lambda,\lambda_1,\lambda_2\in \mathbb{F}-\{0\}\},
$$
for $\sigma$ given in Eq.~\eqref{eq_hola} and $\varepsilon^{T_i}$, $\eta^\lambda$, $\mu^\lambda$ and $\beta^{\lambda_1,\lambda_2}$ the graded contractions described in Section~\ref{subsec_clasificacion}.
The same result  is true by replacing $V$ with $W=\bb F^3$ or with $S=\langle(1,1,1)\rangle^\perp\le W$ and $\sigma$ with the twist in Eq.~\eqref{eq_hola2}.
\end{cor}

\subsection{Some conclusions.}
Only in Corollary~\ref{cor_results}, we already provide  60   different generalized group  algebras which are Lie algebras, distributed in 20 of each dimension between  $16$, $24$ and $32$, 
together with 9 infinite families depending on one or two free parameters,
 again with the aforementioned dimensions. 
 (To be more exact, those families are infinite only if the considered  $\bb F$ is infinite.)
 The provided descriptions permit to multiply easily in these Lie algebras independently of the ground field.
 So, we have shown the potential of the concept of 
generalized group  algebra 
  to provide new examples of Lie algebras.

Thus, a suggestion for describing new Lie algebras with properties is to start with a convenient grading on a possibly well-known Lie algebra with regularity properties on the dimensions of the homogeneous components, and then   study its graded contractions. Perhaps, the results in this work may   seem a coincidence, but even if it were so, there are more coincidences as ours. A convenient candidate for following this study is the exceptional split Lie algebra of dimension 52 of type $F_4$, which is the derivation algebra of an Albert algebra which becomes a twisted group algebra over the group $\bb Z_3^3$.
Graded contractions over this group 
have not been studied so far.

 \bibliographystyle{plain}
\bibliography{Fran_arXiv} 
 \end{document}